\documentclass[11pt]{article}

\usepackage[margin=1in]{geometry}
\usepackage{amsmath,amsthm,amssymb,amsfonts}
\usepackage{mathtools}

\usepackage{graphicx}
\usepackage{enumitem}
\usepackage[hidelinks]{hyperref}

\hypersetup{
colorlinks,
citecolor=black,
filecolor=black,
linkcolor=black,
urlcolor=black
}

\setlist{nolistsep}
\theoremstyle{plain}
\newtheorem{theorem}{Theorem}
\newtheorem{proposition}{Proposition}[section]
\newtheorem{lemma}[proposition]{Lemma}
\newtheorem{corollary}[proposition]{Corollary}
\newtheorem{definition}[proposition]{Definition}
\newtheorem{example}[proposition]{Example}

\newtheorem*{thm:simple-basis}{Theorem  \ref{thm:simple-basis}}
\newtheorem*{thm:semi-fund-basis}{Theorem \ref{thm:semi-fund-basis}}
\newtheorem*{thm:top-ext-seq-alg}{Theorem \ref{thm:top-ext-seq-alg}}
\newtheorem*{thm:top-ext-comp-chain}{Theorem \ref{thm:top-ext-comp-chain}}
\newtheorem*{thm:lin-hull-groups}{Theorem \ref{thm:lin-hull-groups}}
\newtheorem*{thm:lin-hull-vector-sp}{Theorem \ref{thm:lin-hull-vector-sp}}

\newcommand{\script}[1]{\mathcal{#1}}
\newcommand{\bb}[1]{\mathbb{#1}}
\newcommand{\defn}[1]{\textbf{#1}}
\newcommand{\defeq}{\coloneqq}

\newcommand{\Q}{\bb{Q}}

\newcommand{\Z}{\bb{Z}}

\newcommand{\set}[2]{\left\{ {#1} \, : \, {#2} \right\}}
\newcommand{\card}[1]{\left|#1\right|}
\newcommand{\intersect}{\cap}

\newcommand{\union}{\cup}
\newcommand{\Union}{\bigcup}

\DeclareMathOperator{\symdiff}{\Delta}

\newcommand{\cycles}{\script{C}}
\DeclareMathOperator{\ci}{ci}
\DeclareMathOperator{\bo}{bo}

\DeclareMathOperator{\diam}{diam}

\DeclareMathOperator{\lattice}{Lat}
\DeclareMathOperator{\linhull}{Lin.Hull}
\newcommand{\indVect}[1]{\chi_{#1}}
\newcommand{\gpIndex}[2]{\left[{#1} : {#2} \right]}
\newcommand{\basis}{\script{B}}
\newcommand{\tensor}{\otimes}
\newcommand{\Top}{\operatorname{Top}}
\newcommand{\proj}[2]{\left({#2}\right)\hspace{-0.23em}_{#1}}
\newcommand{\clDet}[1]{\det\!\big(\lattice(\cycles({#1}))\big)}

\begin{document}

\title{The Lattice of Cycles of an Undirected Graph}

\author{G.\ Averkov, A.\ Chavez, J.\ A.\ De Loera, B.\ Gillespie}

\maketitle

\begin{abstract}
  We study bases of the lattice generated by the cycles of an undirected graph, defined as the integer linear combinations of the 0/1-incidence vectors of cycles.  We prove structural results for this lattice, including explicit formulas for its dimension and determinant, and we present efficient algorithms to construct lattice bases, using only cycles as generators, in quadratic time.  By algebraic considerations, we relate these results to the more general setting with coefficients from an arbitrary Abelian group.  Our results generalize classical results for the vector space of cycles of a graph over the binary field to the case of an arbitrary field.
\end{abstract}


\section{Introduction}
\label{sec:introduction}

The structure of the cycles of a graph is a rich topic with challenging problems.  Consider for instance three famous covering problems on the set of all cycles of a graph: The \emph{double cover conjecture} states that for any bridgeless graph there exists a list of cycles that contains every edge twice.  \emph{Goddyn's conjecture} further states that if $G$ is a bridgeless graph and $C$ is a cycle in $G$, then there exists a double cover of $G$ containing the cycle $C$.  An \emph{$m$-cycle $k$-cover} is a list of $m$ Eulerian subgraphs covering each edge exactly $k$ times. For example, every bridgeless graph admits a $7$-cycle $4$-cover, but it is an open problem to decide whether every cubic bridgeless graph has a $6$-cycle $4$-cover.
For details see \cite{bermond+jackson+jaeger,diestel_graph_2017,cqzhang} and references therein. Motivated by such \emph{covering and packing problems} using cycles, and relying on the linear structure, this paper studies the lattice generated by the cycles of an undirected connected graph $G$, i.e., the set of all integer linear combinations of $0/1$-incidence vectors of cycles of $G$. We call it the \emph{cycle lattice} of the graph $G$.

Studying the lattices generated by incidence vectors of combinatorial objects is a technique which has been used to model combinatorial problems.  A classical example is of course the case of directed graphs, which has many applications (see the survey \cite{kavitha_cycle_2009}).  The technique has also been used in several cases for undirected graphs, including for matchings, cuts, and cycles (see \cite{goddyn_cones_1991,goddyn_cuts,LLovasz,S78} and references therein).  From these last examples we take inspiration; we provide theoretical and computational results about the bases of the cycle lattice of an undirected connected graph, and some consequences.

In what follows $G = (V, E)$ will denote a connected undirected graph with vertices $V$ and edges $E$.  In general, we allow loops and parallel edges unless otherwise noted.  A \emph{cycle} is a connected subgraph of $G$ with each vertex having degree two, and we write $\cycles(G)$ for the collection of cycles of $G$.  We will usually regard cycles and trees as subsets of $E$.

If $A \subseteq E$, then let $\indVect{A} \in \Z^E$ denote the characteristic vector of $A$.  For a collection $\script{A}$ of subsets of $E$, define the \emph{lattice} of $\script{A}$ by
\[
  \lattice(\script{A}) \defeq \set{\sum_{A \in \script{A}} n_A \indVect{A}}{n_A \in \Z} \subseteq \Z^E.
\]
If $K$ is a field, or more generally an Abelian group, then we define the \emph{$K$-linear hull} of $\script{A}$ to be
\[
  \linhull_K(\script{A})
  \defeq \lattice(\script{A}) \tensor_{\Z} K
  = \set{\sum_{A \in \script{A}} n_A \indVect{A}}{n_A \in K}.
\]

We are interested in studying the properties of these spaces when $\script{A} = \cycles(G)$.  The lattice $\lattice(\cycles(G))$ we call the \emph{cycle lattice} of $G$, and the $\Q$-linear hull $\linhull_{\Q}(\cycles(G))$ we call the \emph{rational cycle space} of $G$.
In Section \ref{sec:cycle-lattice} we give structural and algorithmic results on the cycle lattice of a graph, while in Section \ref{sec:applications-lin-hulls} we explore the consequences of these results on the $K$-linear hulls of graph cycles for different choices of $K$.

The classical \emph{binary cycle space} of $G$ in particular fits into this framework as the linear hull for the choice $K = \Z / 2\Z$.  In our more general setting, we are able to give a dimension formula for arbitrary fields $K$, and we describe a structural characterization of linear hulls for general Abelian groups which sheds light on the special role played in this theory by fields of characteristic 2.
It is also worth mentioning that several authors have considered the case of \emph{directed graphs}, which have different behavior.  For more details on the other linear spaces generated by cycles, see \cite{kavitha_cycle_2009}.

The following result by Goddyn (see \cite[Prop.~2.1]{goddyn_cones_1991}) is our starting point.  It characterizes the rational cycle space via series classes of $E$ (i.e., $e, f \in E$ are \emph{in series} if they are in the same cycles).
\begin{proposition}
  \label{prop:cycle-space-dimension}
  Let $G = (V, E)$ be a graph.  Then the rational cycle space $\linhull_{\Q}(\cycles(G))$ is given by
  \[
    \set{p \in \Q^E}{p_e = 0 \text{ for any bridge } e, \text{ and } p_e = p_f \text{ for } e \text{ and } f \text{ in series}}.
  \]
\end{proposition}

In particular, Proposition \ref{prop:cycle-space-dimension} implies that the rational cycle space and the cycle lattice of $G$ are full-dimensional if and only if $G$ has no bridges and no nontrivial series classes, or equivalently when the graph is 3-edge-connected.
In our analysis we will see that there is no loss of generality in assuming that $G$ is 3-edge-connected (in particular see Lemma~\ref{lem:3-edge-conn-reduction}).

\subsection*{Our Contributions}

Our first main result on the cycle lattice is a key building block utilized throughout the paper.

\begin{theorem}
  \label{thm:simple-basis}
  Let $G = (V, E)$ be a 3-edge-connected graph, and let $T \subseteq E$ be a spanning tree of $G$.  Consider the sets
  \[
    \cycles_T \defeq \set{\indVect{\ci(e, T)}}{e \in E \setminus T} \quad \text{and} \quad X_T \defeq \set{2 \indVect{t}}{t \in T},
  \]
  where $\ci(e, T)$ denotes the unique cycle contained in $T \union e \subset E$.  Then the collection $\cycles_T \union X_T$ is a basis for the cycle lattice of $G$.  Moreover, the determinant of this lattice is given by
  \[
    \det\!\big( \lattice( \cycles(G) ) \big) = 2^{\card{T}}.
  \]
\end{theorem}

The lattice bases provided by this result are a natural extension to the well-known \emph{fundamental cycle bases} of the binary cycle space, but have the disadvantage that they include elements which are not cycles.
With some additional work, we are able to algorithmically produce lattice bases consisting only of cycles.
Note that this is in stark contrast to the lattices of other natural collections such as matchings \cite{LLovasz} and cuts \cite{goddyn_cuts}, which do not always have bases consisting of sets in the generating collection.

A collection of cycles of $G$ is called a \emph{lattice cycle basis} if its indicator vectors form a basis of the cycle lattice.  In Section~\ref{sec:cycle-lattice}, we show the existence of lattice cycle bases of a graph $G$, and we give two algorithmic constructions for such bases.

A lattice cycle basis is called a \emph{semi-fundamental basis} with respect to a spanning forest $F$ if it consists of all of the fundamental cycles of $F$, along with some additional cycles containing exactly two edges outside of $F$, called \emph{semi-fundamental cycles}.
In Section~\ref{subsec:semi-fund-basis-alg}, we describe an efficient algorithm to compute semi-fundamental lattice cycle bases of a graph $G$ with respect to a choice of spanning forest in quadratic time, producing the following.

\begin{theorem}
  \label{thm:semi-fund-basis}
  Let $G$ be a connected graph with $m$ edges and $n$ vertices.  Then a lattice cycle basis of $G$ exists, and can be constructed in time $O(mn)$.  If $T$ is any spanning tree of $G$, then the basis may be chosen to be semi-fundamental with respect to $T$.
\end{theorem}

A potentially useful property of the semi-fundamental bases given by this algorithm is that all cycles included have length bounded by $2 \diam(T)$; see Corollary \ref{cor:bdd-cycle-lengths}.  In particular, this bound may be quite nontrivial for common classes of graphs with high connectivity such as expander graphs and small-world networks \cite{bollobas_diameter_2004,hoory_expander_2006}.
Several papers investigate the variety of different cycle lengths possible in a graph, for example \cite{MR1692279,MR766494,MR2096823}, and our length bound indicates that in many cases, the cycles in a lattice basis may be chosen from a very restricted set.  Additionally, lattice generators with sparse support are relevant to applications in several areas of optimization; see \cite{aliev2020optimizing} and references.

If $G, H$ are graphs, then $G$ is called a \emph{topological one-edge extension} of $H$ if it is obtained from $H$ by connecting two vertices, either existing in $H$ or created by dividing edges of $H$ in two, by a new edge.  (See Definition \ref{def:top-one-edge-ext}.)
If $G$ is 3-edge-connected, then a sequence of topological one-edge extensions starting at the single-vertex graph and ending at $G$ is called a \emph{topological extension sequence} of $G$.  In particular, a graph $G$ is known to be 3-edge-connected if and only if it admits a topological extension sequence, and we present an algorithm to produce such a sequence.

\begin{theorem}
  \label{thm:top-ext-seq-alg}
	Let $G$ be a 3-edge-connected graph with $n$ vertices and $m$ edges.  Then a topological extension sequence for $G$ exists, and can be constructed in time $O(mn)$.
\end{theorem}

If $G_0, G_1, \ldots, G_k$ is a topological extension sequence of a graph $G$, then a nested sequence $(\cycles_i)$ with $\cycles_i$ a lattice cycle basis of $G_i$ is called a \emph{compatible chain} of lattice cycles bases.  In Section \ref{subsec:top-ext-basis-alg}, we give a different algorithm for a basis of the cycle lattice which produces a topological extension sequence and a compatible chain of lattice cycle bases.

\begin{theorem}
  \label{thm:top-ext-comp-chain}
  Let $G$ be a 3-edge-connected graph with $m$ edges and $n$ vertices.  Then a topological extension sequence of $G$ and a compatible chain of lattice cycle bases can be constructed in time $O(mn)$.
\end{theorem}

The construction of the above algorithm can additionally be extended to general connected graphs using the reduction of Lemma \ref{lem:3-edge-conn-reduction}.

Finally, in Section~\ref{sec:applications-lin-hulls} we relate the cycle lattice of a graph to the $A$-linear hull for $A$ an Abelian group.  The main structural result is given by the following.
\begin{theorem}
  \label{thm:lin-hull-groups}
  Let $G = (V, E)$ be a 3-edge-connected graph with $m$ edges and $n$ vertices, and let $A$ be an Abelian group.  Then
	\[
		\linhull_A(\cycles(G)) \simeq (2A)^{n-1} \oplus A^{m-n+1}.
	\]
\end{theorem}

This result is applied in Theorem \ref{thm:lin-hull-vector-sp} to the case when $A$ is a field, generalizing known results about the classical binary cycle space:
\begin{theorem}
  \label{thm:lin-hull-vector-sp}
  Let $G = (V, E)$ be a 3-edge-connected graph with $m$ edges and $n$ vertices, and let $K$ be a field of characteristic $p$.  Then $\linhull_K(\cycles(G))$ is a $K$-vector space of dimension
	\[
		\dim_K\!\big(\linhull_K(\cycles(G))\big)
    =
		\begin{cases}
			m, & \text{if } p \ne 2,
			\\ m-n+1, & \text{if } p = 2.
		\end{cases}
	\]
  If $p \neq 2$, then any lattice basis of $\lattice(\cycles(G))$ reduces modulo $p$ to a linear basis of $\linhull_K(\cycles(G))$.  If $p = 2$, then any basis of the classical binary cycle space maps to a linear basis of $\linhull_K(\cycles(G))$ under the natural inclusion map.
\end{theorem}

The remainder of the paper is organized as follows.  In Section~\ref{sec:preliminaries} we give a brief overview of relevant background material and prior results in graph theory, and discuss preliminary computational results and our computational model.
In Section~\ref{subsec:lattice-basics} we study the basic structure of the cycle lattice and derive Theorem~\ref{thm:simple-basis},
and in Sections~\ref{subsec:semi-fund-basis-alg} and \ref{subsec:top-ext-basis-alg}
we present two approaches for producing lattice cycle bases of graphs, in particular proving Theorems~\ref{thm:semi-fund-basis}, \ref{thm:top-ext-seq-alg} and \ref{thm:top-ext-comp-chain}.
In Section~\ref{sec:applications-lin-hulls} we summarize several consequences of our results for linear hulls of cycles with respect to fields and Abelian groups, and we give proofs of Theorems~\ref{thm:lin-hull-groups} and \ref{thm:lin-hull-vector-sp}.

\textbf{Acknowledgements:} We thank Prof.\ Andras Frank for his comments about the algorithmic aspects of this problem.  We acknowledge Mr.\ Yuanbo Li for his work, during a Summer REU, on an earlier version of this paper regarding the relationship of the lattice of cycles and the Tutte decomposition.  We thank the reviewers and editors for their help and feedback.  The first author is supported by the DFG (German Research Foundation) within the project number 413995221.  The second author was partially supported by NSF grant DMS-1802986.  The third author gratefully acknowledges support from NSF grant DMS-1818969.

\section{Preliminaries}
\label{sec:preliminaries}

In the following, we introduce preliminary material which will be needed throughout the remainder of the work.  First we give a brief overview of the ideas in graph theory that will be assumed as background.
After this, we discuss the computational model that will be used for our algorithmic assertions, and we develop a computational reduction, Lemma \ref{lem:3-edge-conn-reduction}, which will be essential for our analysis in Section \ref{sec:cycle-lattice}.

\subsection{Graph Theory}

We briefly provide the graph theory concepts necessary to understand our results.  For standard background from graph theory, including basic definitions and notations, we refer to \cite{diestel_graph_2017}.

Here and throughout this work, graphs are allowed to have loops and parallel edges.  Let $G = (V, E)$ be an undirected graph with vertices $V$ and edges $E$.  If $v$ is a vertex of $G$, then its degree $\deg(v)$ is the number of non-loop edges incident to $v$ plus twice the number of loops incident to $v$.  A (simple) path $P$ of length $k$ in $G$ is a subgraph of $G$ with distinct vertices $\{x_0, x_1, \ldots, x_k\}$ and edges $\{x_0 x_1, x_1 x_2, \ldots, x_{k-1} x_k\}$.
A (simple) \defn{cycle} of length $k+1$ in $G$ is a subgraph of $G$ consisting of a simple path along with an additional edge connecting its endpoints.

If $e \in E$, then the \defn{deletion} of $e$ from $G$ is the graph $G \setminus e$ obtained from $G$ by removing the edge $E$, and the \defn{contraction} of $e$ in $G$ is the graph $G / e$ obtained from $G$ by combining its endpoints into a single vertex, and removing $e$ from the result.  The cycles of $G \setminus e$ are exactly the cycles of $G$ which do not contain $e$, and the cycles of $G / e$ are the inclusion-minimal nonempty subgraphs within the set of graphs $\set{C / e}{C \text{ a cycle of } G}$.

If $E_1, E_2 \subseteq E$ are disjoint sets of edges, then a graph may be obtained by deleting the edges of $E_1$ and contracting the edges of $E_2$ in any order.  The resulting graph is independent of the order chosen, and is denoted $G \setminus E_1 / E_2$.  A graph which can be obtained in this way from $G$ is called a \defn{minor} of $G$.

An edge $e \in E$ is called a \defn{bridge} if $e$ is contained in no cycle of $G$, and edges $e, f \in E$ are said to be in \defn{series} if $e \in C$ implies $f \in C$ for every cycle $C$.  The relation of being in series is an equivalence relation on $E$ whose equivalence classes are called the \defn{series classes} of $G$.  A series class is called \defn{non-trivial} if it has more than one element, and \defn{trivial} otherwise.

The graph $G$ is \defn{connected} if it contains a path between any two vertices, and it is \defn{$\boldsymbol{k}$-edge-connected} if $G \setminus E_1$ is connected for any set $E_1$ of $k-1$ or fewer edges.  Most importantly for our purposes, a connected graph is 2-edge-connected if and only if it has no bridges, and is 3-edge-connected if and only if it has no bridges and no nontrivial series classes.

Edge connectivity of a graph is related to the following version of the well-known \emph{Menger's theorem} (see for instance \cite[Sec.~3.3]{diestel_graph_2017}), which is critical to our proofs.

\begin{proposition}[Menger's theorem, edge version]
  Let $G$ be an undirected graph and let $u, v \in V(G)$ be distinct vertices.  Then the minimum number of edges which can be deleted from $G$ to disconnect $u$ and $v$ is equal to the maximum number of edge-disjoint paths connecting $u$ and $v$.
\end{proposition}

In particular, if $G$ is $k$-edge-connected and $u, v \in V$, then there exist $k$ edge-disjoint paths connecting $u$ and $v$.

Suppose $G$ is connected, and let $T \subseteq E$ be a spanning tree of $G$.  Recall that for each edge $e \in E \setminus T$, there is a unique cycle contained in the edge set $T \union e$, which is called the \defn{fundamental cycle} of $e$ with respect to $T$ and is denoted $\ci(e, T)$.
For each edge $t \in T$, the forest $T \setminus t$ has two connected components, which induces a cut of $G$ between the corresponding vertex sets.  This cut is called the \defn{fundamental cut} or \defn{fundamental bond} of $t$ with respect to $T$, and is denoted $\bo(t, T)$.
Fundamental cycles and fundamental cuts exhibit the following duality (see e.g., \cite[Lem.~7.3.1]{bjorner_homology_1992}).
\begin{lemma}
  \label{lem:ci-bo-duality}
  Let $G = (V, E)$ be a connected graph and let $T \subseteq E$ be a spanning tree.  If $t \in T$ and $e \in E \setminus T$, then
  \[
    e \in \bo(t, T) \text{ if and only if } t \in \ci(e, T).
  \]
\end{lemma}

The well-studied \defn{binary cycle space} of $G$ is defined as $\lattice(\cycles(G)) \tensor (\Z / 2 \Z)$, and can be thought of as the vector subspace of $(\Z / 2 \Z)^E$ spanned by the indicator vectors of $\cycles(G)$.  It is known that the collection of all fundamental cycles with respect to a fixed spanning tree of $G$ gives a basis of this space; the following result summarizes this and other related properties.

\begin{proposition}[{\cite[Sec.~1.9]{diestel_graph_2017}}]
  \label{prop:classical-cycle-space}
  Let $G = (V, E)$ be a connected graph with binary cycle space $\script{B}$.  Then:
  \begin{itemize}
    \item $\script{B}$ is the collection of characteristic vectors of Eulerian subgraphs of $G$.

    \item If $T \subseteq E$ is a spanning tree of $G$, then $\set{\indVect{\ci(e,T)}}{e \in E \setminus T}$ is a basis of $\script{B}$.

    \item The dimension of $\script{B}$ is $\card{E} - \card{V} + 1$.
  \end{itemize}
\end{proposition}

\subsection{Computational Model and Graph Reductions}

The structure of graph cycle lattices can in many cases be reduced to the case of 3-edge-connected graphs.  In the following we give details of this reduction, and describe the computational model we use for algorithmic complexity bounds.
The key result connecting cycle lattices of 3-edge-connected graphs with those of general graphs is found in Lemma \ref{lem:3-edge-conn-reduction}.

Let $G = (V, E)$ be a graph with $m$ edges and $n$ vertices.  We will assume that $V$ is a totally ordered set, and that $G$ is represented as an ordered adjacency list.  In particular, inspecting the edges of $G$ adjacent to a vertex $v$ takes time $O(\deg(v))$, and inspecting all of the edges of $G$ takes time $O(m)$.  If $E_1, E_2 \subseteq E$ are disjoint, then the graph minor $G \setminus E_1 / E_2$ can be computed in time $O(mn)$ by deleting all of the edges of $E_1 \union E_2$ from $G$ and merging the remaining adjacencies of vertices connected by a path in $E_2$.

\begin{lemma}
  \label{lem:fund-cycles-matrix}
  Let $G = (V, E)$ be a connected graph with $m$ edges and $n$ vertices, and let $T$ be a spanning tree of $G$.  Then the fundamental cycles and fundamental cuts of $T$ can be computed in time $O(mn)$.
\end{lemma}

\begin{proof}
  We will record the fundamental cycles of $T$ by computing the $T \times (E \setminus T)$ binary matrix $X$ with values
  \[
    X_{t,e} = \begin{cases}
      1, & t \in \ci(e, T) \\
      0, & \text{otherwise}
    \end{cases}.
  \]
  This simultaneously computes the fundamental cycles and the fundamental cuts of $T$ because $t \in \ci(e, T)$ if and only if $e \in \bo(t, T)$ by Lemma \ref{lem:ci-bo-duality}.

  The algorithm proceeds as follows.  Initialize $X$ to all zeros, and pick an arbitrary vertex $v_0$ of $T$.  Traverse $T$ to compute for each vertex $v \in V$ the path $P_v$ in $T$ from $v_0$ to $v$.  For each edge $e \in E \setminus T$, let $v, v'$ be its endpoints.  Find the first edge at which $P_v$ and $P_{v'}$ differ, and set $X_{t,e}$ to $1$ for the subsequent edges of these paths.

  The elements of $T$ in a fundamental circuit $\ci(e, T)$ are given by the unique path in $T$ between the endpoints of $e$, so we see that the edges recorded in this way represent the fundamental circuits of $T$ as desired.  The computational time $O(mn)$ follows because each path $P_v$ has at most $n - 1$ edges, and the number of edges $e \in E \setminus T$ is bounded by $m$.
\end{proof}

\begin{lemma}
  \label{lem:bridges-series-classes}
  Let $G = (V, E)$ be a connected graph with $m$ edges and $n$ vertices.  Then the bridge elements and series classes of $G$ can be computed in time $O(mn)$.
\end{lemma}

\begin{proof}
  Let $T$ be a spanning tree of $G$, and let $X$ be the $T \times (E \setminus T)$ fundamental cycle matrix from Lemma \ref{lem:fund-cycles-matrix}, which can be computed in time $O(mn)$.  The bridge elements of $G$ are those which appear in no fundamental cycle of $T$, and so can be identified as the indices of the all-zero rows of $X$.

  For the series classes, note that if $x, y \in E$ are in different series classes, then there are elements of the binary cycle space for which the $x$ and $y$ coordinates differ.  In particular, since the fundamental cycles of $T$ generate the binary cycle space, this implies that there is a \emph{fundamental} cycle of $T$ containing one of $x, y$ but not the other.  Thus to compute the series classes of $G$, it is sufficient to compute the partition of $E$ corresponding to the common refinement of the partitions $\{C, E \setminus C\}$ where $C$ is a fundamental cycle with respect to $T$.  Because each cycle $C$ has length at most $n$ and there are $m - n + 1 = O(m)$ such cycles, this refinement can be computed in time $O(mn)$.
\end{proof}

\begin{lemma}
  \label{lem:series-contractions}
  Let $G = (V, E)$ be a connected graph, and let $e \in E$.  Then:
  \begin{enumerate}
    \item \label{p:series-contraction} If $e$ is contained in a non-trivial series class of $G$, then $\cycles(G)$ and $\cycles(G / e)$ are in bijection by the map $C \mapsto C / e$.
    \item \label{p:bridge-deletion} If $e$ is a bridge of $G$, then $\cycles(G) = \cycles(G \setminus e)$.
  \end{enumerate}
  Let $\pi : \Z^E \to \Z^{E \setminus e}$ be the standard projection map.  In both of the cases above, $\pi$ induces a lattice isomorphism between the cycle lattice of $G$ and the cycle lattice of the corresponding graph minor.
\end{lemma}

\begin{proof}
  Part \ref{p:series-contraction} follows because the cycles of $G / e$ are the nonempty subgraphs in $\set{C / e}{C \in \cycles(G)}$ which are minimal under inclusion of edge sets, and Part \ref{p:bridge-deletion} follows because no cycle contains a bridge, and the cycles of $G \setminus e$ are those of $G$ not containing $e$.
  The projection $\pi$ induces a lattice isomorphism in each case because it maps the indicator vector of a cycle in $G$ to the indicator vector of the corresponding cycle in $G / e$ or $G \setminus e$.
\end{proof}

If $G$ is a connected graph, let $\hat{G}$ denote a graph obtained from $G$ by deleting all bridges of $G$ and contracting all but one element from each nontrivial series class of edges in $G$.  The graph $\hat{G}$ is called a \defn{cosimplification} of $G$.  This induces a projection map $\pi : E(G) \to E(\hat{G}) \union \{\epsilon\}$, where $\epsilon$ is a formal symbol disjoint from $E(G)$, given by
\[
  \pi : e \mapsto
  \begin{cases}
    \epsilon, & e \text{ a bridge element} \\
    \hat{e}, & \text{otherwise}
  \end{cases},
\]
where $\hat{e}$ denotes the representative of the series class in $G$ of an edge $e$ in $\hat{G}$.  The connected components of $\hat{G}$ are in particular 3-edge-connected, and this gives a graph reduction which is useful in studying cycles and cycle bases.

\begin{lemma}
  \label{lem:3-edge-conn-reduction}
  Let $G = (V, E)$ be a connected graph with $m$ edges and $n$ vertices.  Then a cosimplification $\hat{G}$ can be computed in time $O(mn)$.  Let $G_1, \ldots, G_k$ be the connected components of $\hat{G}$, and let $\cycles_i$ be a lattice cycle basis of $G_i$ for each $i$.  Then a lattice cycle basis of $G$ is given by $\cycles = \Union_i \set{\pi^{-1} (C)}{C \in \cycles_i}$, which can be computed in time $O(mn)$.
\end{lemma}

\begin{proof}
  From Lemma \ref{lem:bridges-series-classes}, we can compute the bridges and series classes of $G$ in time $O(mn)$, from which the graph minor $\hat{G}$ can be constructed.

  Lemma \ref{lem:series-contractions} implies that $\pi$ induces a lattice isomorphism from $\lattice(\cycles(G))$ to $\lattice(\cycles(\hat{G}))$.  Since the cycle lattice of a graph admits a direct sum decomposition over connected components, the collection $\Union_i \cycles_i$ yields a lattice cycle basis of $\hat{G}$.  Thus, $\pi^{-1}$ lifts to a lattice cycle basis of $G$.

  The computation of $\pi^{-1}$ can be accomplished by checking each cycle for the presence of the series class representative for a non-trivial series class and extending, if the representative is found, to include the whole series class in $G$.
  There are $O(m)$ cycles in the basis of $\hat{G}$, and $O(n)$ nontrivial series classes, yielding a time bound of $O(mn)$ for the computation.
\end{proof}

\section{The Cycle Lattice of a Graph}
\label{sec:cycle-lattice}

We now develop results relating to the cycle lattice of a graph, with an emphasis on understanding the structure of and algorithms for producing lattice bases.
In Section \ref{subsec:lattice-basics}, we prove basic structural results of the cycle lattice, and produce a simple lattice basis extending the fundamental cycle basis of the classical binary cycle space.
In Section \ref{subsec:semi-fund-basis-alg} we present an algorithm to produce \emph{semi-fundamental} lattice cycle bases, and in Section \ref{subsec:top-ext-basis-alg} we present a different algorithm which sequentially expands lattice cycle bases for the graph minors in a \emph{topological extension sequence}.
Both algorithms will be seen to produce a lattice cycle basis in time $O(mn)$, where $m$ is the number of edges of the graph and $n$ is the number of vertices.

\subsection{Lattice Structure and a Non-cycle Basis}
\label{subsec:lattice-basics}

As a first step toward understanding the lattice structure, we make the following observation.

\begin{lemma}
  \label{lem:twice-edges-in-lattice} Let $G$ be a 3-edge-connected graph.
  For each $e \in E$, the vector $2 \indVect{e}$ is an element of $\lattice(\cycles(G))$.  In particular, $2 \Z^E \subseteq \lattice(\cycles(G))$.
\end{lemma}

\begin{proof}
  Let $e \in E$, and without loss of generality suppose $e$ connects distinct vertices $u$ and $v$.  Since $G$ is 3-edge-connected, any minimal cut in $G \setminus e$ disconnecting vertices $u$ and $v$ contains at least two edges.  By Menger's theorem, there are edge-disjoint simple paths $P$ and $Q$ between $u$ and $v$ which exclude the edge $e$.

  In particular, $P \union e$ and $Q \union e$ are cycles whose only common edge is $e$.  Additionally, $P \union Q$ is a (potentially non-simple) cycle of $G$, which can be written as a disjoint union of cycles, $P \union Q = C_1 \union \cdots \union C_k$.  From this we obtain that
  \[
    2 \indVect{e} = \indVect{P \union e} + \indVect{Q \union e} - \sum_i \indVect{C_i}
  \]
  is in $\lattice(\cycles(G))$.
\end{proof}

As a first corollary, we obtain a basis-free description of the cycle lattice, analogous to Proposition~\ref{prop:cycle-space-dimension}.

\begin{corollary}
  Let $G = (V, E)$ be a graph.  Then $\lattice(\cycles(G))$ is given by the collection of $p \in \Z^E$ such that
  \begin{itemize}
    \item $p_e = 0$ for any bridge $e$.
    \item $p_e = p_f$ for $e$ and $f$ in series.
    \item $\sum_{e \in N(v)} p_e$ is even for each vertex $v$, where $N(v)$ is the neighborhood of edges incident to $v$.
  \end{itemize}
\end{corollary}

\begin{proof}
  The indicator vector of any cycle of $G$ satisfies all of the above conditions, so this is likewise true for any element of the cycle lattice.  Now suppose $p \in \Z^E$ satisfies the conditions above, and let $E_1 = \set{e \in E}{p_e \text{ is odd}}$.  By the parity condition,
  \[
    0 \equiv \sum_{e \in N(v)} p_e \equiv \sum_{e \in N(v) \cap E_1} p_e \equiv | N(v) \cap E_1| \pmod{2}.
  \]
  Since each vertex $v$ is incident to an even number of edges of $E_1$, this implies that $E_1$ induces an Eulerian subgraph of $G$, which thus can be decomposed into a disjoint union of a collection $\cycles'$ of cycles.  Letting $p' = p - \sum_{C \in \cycles'} \indVect{C}$, we see that $p'$ has only even coordinates.

  Passing to a cosimplification $\hat{G}$ of $G$, the projection of $p'$ lies in the cycle lattice of $\hat{G}$ by Lemma \ref{lem:twice-edges-in-lattice}.
  Because $p'$ is zero on bridges and equal on all elements of a series class, the cycle decomposition of the projection lifts to a cycle decomposition of $p'$ in $G$.  Thus, $p'$ lies in the cycle lattice of $G$, and consequently so does $p$.
\end{proof}

The fact that $2 \Z^E$ is a sublattice of $\lattice(\cycles(G))$ for 3-edge-connected $G$ means that we can view the quotient $\lattice(\cycles(G)) / 2 \Z^E$ as a subspace of the binary vector space $(\Z / 2 \Z)^E$.  This allows us to directly compute the determinant of the lattice.

\begin{proposition}
  \label{prop:cycle-lattice-det}
  Let $G = (V, E)$ be a 3-edge-connected graph with cycle lattice $\script{L} = \lattice(\cycles(G))$.  Then $\det(\script{L}) = 2^{\card{V} - 1}$.
\end{proposition}

\begin{proof}
  The determinant of $\script{L}$ can be expressed in terms of group indices as $\gpIndex{\Z^E}{\script{L}}$, which implies
  \[
  \frac{2^{\card{E}}}{\det(\script{L})} = \gpIndex{\script{L}}{2 \Z^E} = 2^{\dim_{\Z/2\Z} \big( \script{L} / 2 \Z^E \big)}.
  \]
  Here, the first equality is by comparison of determinants, and the second is by interpreting the lattice quotient $\script{L} / 2 \Z^E$ as a vector subspace of $\Z^E / 2\Z^E$.  By Proposition \ref{prop:classical-cycle-space}, the space of cycles over $\Z / 2 \Z$ has dimension $\card{E} - \card{V} + 1$, from which we can compute $\det(\script{L})$ directly.
\end{proof}

With explicit formulas for the dimension and the determinant of $\lattice(\cycles(G))$, we are now able to produce an explicit lattice basis by producing an appropriate number of lattice vectors with the correct determinant.  This is accomplished by extending the collection of indicator vectors of fundamental cycles $\ci(e, T)$ with respect to a fixed spanning tree $T$, as follows.

\begin{proposition}
  \label{prop:simple-basis}
  If $G = (V, E)$ is a 3-edge-connected graph and $T$ is a spanning tree of $G$, let
  \[
    \cycles_T \defeq \set{\indVect{\ci(e, T)}}{e \in E \setminus T} \quad \text{and} \quad X_T \defeq \set{2 \indVect{t}}{t \in T}.
  \]
  Then the collection $\cycles_T \union X_T$ is a basis for $\lattice(\cycles(G))$.
\end{proposition}

We present two short proofs of this fact, one of which directly applies our computation for the determinant of $\lattice(\cycles(G))$, and a second of independent interest which uses the duality between fundamental cycles and fundamental cuts.

\begin{proof}[Proof 1]
  We have that $X_T \subseteq \lattice(\cycles(G))$ by Lemma \ref{lem:twice-edges-in-lattice}.  Note that the vectors in $\cycles_T \union X_T$ are naturally identified with the edges of $G$.  Then for a fixed ordering on $E$ for which the elements of $T$ come before the elements of $E \setminus T$, the matrix of column vectors of $\cycles_T \union X_T$ induced by this ordering on both the rows and columns can be seen to have the following block structure,
  \[
    \begin{bmatrix}
      2 I_{\card{T}} & A \\
      0 & I_{\card{E \setminus T}}
    \end{bmatrix},
  \]
 where $I_k$ is the $k \times k$ identity matrix, and $A$ is the matrix whose columns are the indicator vectors of $\ci(e, T) \setminus e$ for $e \in E \setminus T$.  From this block structure we see that the determinant of these lattice vectors is $2^{\card{V} - 1}$. Thus, by Proposition \ref{prop:cycle-lattice-det}, the collection forms a lattice basis of $\lattice(\cycles(G))$, as desired.
\end{proof}

Our second proof of Proposition \ref{prop:simple-basis} relies on the notion of fundamental bonds we discussed earlier.

\begin{proof}[Proof 2]
  We have that $X_T \subseteq \lattice(\cycles(G))$ by Lemma \ref{lem:twice-edges-in-lattice}.
  Since the lattice has dimension $\card{E} = \card{\cycles_T \union X_T}$ by Proposition \ref{prop:cycle-space-dimension}, it is sufficient to show that $\cycles_T \union X_T$ generates any cycle of $G$.

  For this, let $p = \chi_C$ be the indicator vector of a cycle in $G$, and note that since the intersection of a cycle with a cut has even cardinality, we have
  \[
    \sum_{e \in D} p_e \equiv 0 \pmod{2},
  \]
  for any cut $D$ of $G$.  Letting $q = \sum_{e \in E \setminus T} p_e \indVect{\ci(e, T)}$, we see that for each $e \in E \setminus T$ the $e$-components of $q$ and $p$ are equal.  Further, for each $t \in T$, the $t$-component of $q$ can be written as
  \[
    q_t
    = \sum_{\substack{e \in E \setminus T \\ \text{s.t.\,} t \in \ci(e, T)}} \hspace{-0.7em} p_e
    = \sum_{e \in \bo(t, T)} \hspace{-0.5em} p_e \hspace{0.3em} - \hspace{0.2em} p_t.
  \]
  Denoting the sum in the latter expression above by $S_t$, note that $S_t + p_t$ is the sum of the components of $p$ across the fundamental cut $\bo(t, T)$, which is even.  This implies that $S_t - p_t$ is even.  Letting $\alpha_t = (S_t - p_t) / 2$, we have
  \[
    p = q + \sum_{t \in T} \alpha_t (2 \indVect{t}).
  \]
  Thus $p = \indVect{C}$ is generated by $\cycles_T \union X_T$, and this concludes the proof.
\end{proof}

Note that while the above material is formulated in the setting of graphs, the results may be extended to the more general class of binary matroids satisfying the \defn{lattice of circuits} property (see \cite[Sec.~2]{goddyn_cones_1991}).  We omit these details here, but the arguments involved are substantially similar.

Proposition~\ref{prop:simple-basis} provides a lattice basis of $\lattice(\cycles(G))$ which is useful for many applications, but includes elements outside of the generating collection of cycle indicator vectors.  We next consider lattice bases consisting only of cycle indicator vectors.

\begin{definition}
  If $\cycles'$ is a collection of cycles of a graph $G$ whose indicator vectors form a basis of the lattice $\lattice(\cycles(G))$, we call $\cycles'$ a \defn{lattice cycle basis} of $G$.
\end{definition}

We emphasize that in the above definition, the additive structure is over $\Z$ rather than over $\Z / 2\Z$ as in the classical binary cycle space, so that a priori it is not clear if a lattice cycle basis of a graph always exists.  Indeed, in the context of lattices, a generating set does not always contain a basis.  For example, the set $\{2,3\}$ generates the lattice $\Z$, but $\{2,3\}$ contains no basis of $\Z$.
Somewhat surprisingly, such a phenomenon never occurs for the generator set $\set{\indVect{C}}{C \in \cycles(G)}$ of the lattice $\lattice(\cycles(G))$.  In the following sections, we provide two constructions for such lattice cycle bases.

\subsection{Semi-fundamental Lattice Cycle Bases}
\label{subsec:semi-fund-basis-alg}

We now describe an algorithm to produce a cycle basis using the fundamental cycles of a spanning tree, and some additional cycles which we call \emph{semi-fundamental}.

\begin{definition}
Let $G = (V, E)$ be a connected graph, let $T$ be a spanning tree of $G$, and let $e, f \in G \setminus T$, $e \ne f$. If $\ci(e,T)$ and $\ci(f,T)$ intersect in at least one edge of $T$, we call the symmetric difference
\[
	\ci(e f, T):= \ci(e,T) \symdiff \ci(f,T)
\]
a \defn{semi-fundamental cycle} of $G$ with respect to the tree $T$.

A lattice cycle basis $\cycles'$ of $G$ is called \defn{semi-fundamental} with respect to $T$ if $\cycles'$ contains all of the fundamental cycles with respect to $T$, and all other cycles of $\cycles'$ are semi-fundamental with respect to $T$.
\end{definition}

The following lemma provides a key inductive step for the subsequent algorithm.

\begin{lemma}
  \label{lem:fund-intersection}
  Let $G = (V, E)$ be a 3-edge-connected graph with at least two vertices, and let $T$ be a spanning tree of $G$. Then there exist two fundamental cycles of $T$ that intersect in a single edge of $T$.
\end{lemma}

\begin{proof}
  We first prove that if there exist two fundamental cycles $C, C'$ whose intersection is a path $P$ of positive length $k$, then there exist two fundamental cycles whose intersection is a single edge of this path.  If $P$ has length 1 then $C$ and $C'$ are already sufficient to conclude, so suppose that $P$ has length at least 2.

  Let $v_0, \ldots, v_k$ be the vertices of $P$ occurring on $P$ in this order, let $e$ be the edge joining $v_0$ with $v_1$ and $e'$ the edge joining $v_{k-1}$ with $v_k$. Since $e \ne e'$, the forest $T \setminus \{e, e'\}$ has three connected components.  Let $H_P$ be the component which contains $P \setminus \{e, e'\}$, and let $H$ and $H'$ be the connected components containing the vertices $v_0$ and $v_k$, respectively.

  Since $G$ is 3-edge-connected, $G \setminus \{e, e'\}$ is connected, and some edge $x \notin T$ connects $H_P$ to either $H$ or $H'$.  Without loss of generality, assume that $x$ connects $H_P$ to $H$.  In particular, the fundamental cycle $C'' = \ci(x, T)$ contains $e$ and avoids $e'$. Consequently, $P \intersect C''$ is a nonempty proper subpath of $P$ that contains the edge $e$ and does not contain $e'$.

  Since $C$ and $C'$ diverge at $v_0$, the edges of $C \setminus P$ and $C' \setminus P$ incident to $v_0$ are distinct.  In particular, at least one of these edges is distinct from the edge of $C'' \setminus P$ incident to $v_0$.  So, suppose without loss of generality $C$ and $C''$ diverge at $v_0$.
  Since $e'$ is in $C$ but not in $C''$, there is another vertex $v_j$ with $1 \leq j < k$ at which $C$ and $C''$ diverge.  Consequently, the intersection $P' = C \intersect C''$ of fundamental cycles $C$ and $C''$ is a proper sub-path of $P$ of length strictly less than $k$.

  To conclude the lemma, note that if $t \in T$, then the 3-edge-connectivity of $G$ implies that $\bo(t, T)$ contains at least two distinct edges $e_1, e_2 \in E \setminus T$.  In particular $\ci(e_1, T) \intersect \ci(e_2, T) \ni t$ is a nonempty path of $T$.  Thus there exist two fundamental cycles of $T$ with nonempty intersection, so by reverse induction we conclude that there exist fundamental cycles sharing exactly one edge.
\end{proof}

We next prove Theorem \ref{thm:semi-fund-basis}, giving an algorithm to efficiently produce a semi-fundamental lattice cycle basis.

\begin{thm:semi-fund-basis}
	Let $G$ be a connected graph with $m$ edges and $n$ vertices.  Then a lattice cycle basis of $G$ exists, and can be constructed in time $O(mn)$.  If $T$ is any spanning tree of $G$, then the basis may be chosen to be semi-fundamental with respect to $T$.
\end{thm:semi-fund-basis}

\begin{proof}
  Assume first that $G = (V, E)$ is 3-edge-connected.  We inductively construct sequences
  \begin{equation}
  	\label{eq:three-sequences}
  	\big(t_k \big)_{k \in \{1, \ldots, n-1\}},  \ \big( e_k \big)_{k \in \{1, \ldots, n-1\}}, \ \big( f_k \big)_{k \in \{1, \ldots, n-1\}},
  \end{equation}
  where the sequence $(t_k)_{k \in \{1, \ldots, n-1\}}$ gives an ordering of the edges of $T$, while $e_k$ and $f_k$ are edges outside of $T$.  Suppose that $t_1, \ldots, t_{i-1}$ have already been constructed, and let
  \begin{align}
  	\label{eq:Gk:Tk}
  	G_i & \defeq G / \{t_1, \ldots, t_{i-1}\} &  & \text{and} & T_i & \defeq T / \{t_1, \ldots, t_{i-1}\}.
  \end{align}

  In particular, $G_i$ is 3-edge-connected since this property is preserved by graph contraction, and $T_i$ is a spanning tree of $G_i$.  By Lemma~\ref{lem:fund-intersection}, there exist edges $e_i, f_i \in G_i \setminus T_i = G \setminus T$ such that the fundamental cycles of $e_i$ and $f_i$ with respect to $T_i$ in $G_i$ satisfy
  \[
  	\ci(e_i, T_i) \intersect \ci(f_i, T_i) = \{t_i\}
  \]
  for some $t_i \in T_i$.  We will show that the pairs of edges $(e_i, f_i)$ define semi-fundamental cycles of $T$ which produce a semi-fundamental lattice cycle basis.

  The fundamental cycles of $T$ are preserved under tree contractions, in the sense that
  \[
    \ci(e, T_k) = \ci(e,T) / \{t_1, \ldots, t_{k-1}\}
  \]
  holds for every $k$ and every $e \in E \setminus T$.  Consequently, $t_k$ is a common edge of the cycles $\ci(e_k, T)$ and $\ci(f_k, T)$, which implies that the semi-fundamental cycles $\ci(e_k f_k, T)$ are well-defined.  Thus, let
  \[
    \cycles' \defeq \set{\ci(e, T)}{e \in G \setminus T} \union \set{\ci(e_k f_k, T)}{k \in \{1, \ldots, n-1\}},
  \]
  and let $\Lambda$ be the lattice generated by $\cycles'$.  For $A_k \defeq \ci(e_k, T) \intersect \ci(f_k, T)$, we have
  \[
  	2 \indVect{A_k} = \indVect{\ci(e_k, T)} + \indVect{\ci(f_k, T)} - \indVect{\ci(e_k f_k, T)} \in \Lambda.
  \]
  By construction, $t_k \in A_k \subseteq \{t_1, \ldots, t_k\}$.  Thus, $2 \indVect{t_k} = 2 \indVect{A_k} - \sum_{t \in A_k \setminus \{t_k\}} 2 \indVect{t}$.  By induction on $k$ we see that $2 \indVect{t_k} \in \Lambda$ for every $k \in \{1, \ldots, n-1\}$.

  It follows that $\Lambda$ contains the basis of $\lattice(\cycles(G))$ from Proposition~\ref{prop:simple-basis}, and consequently, $\Lambda = \lattice(\cycles(G))$.  Taking into account that $\cycles'$ consists of $m = \dim(\lattice(\cycles(G)))$ cycles, we conclude that $\cycles'$ is a lattice cycle basis of $G$.

  To verify the algorithmic part of the assertion, we need to expand on the procedure of constructing $\cycles'$ suggested above.  The procedure relies on the constructive proof of Lemma~\ref{lem:fund-intersection}, so we first explain how to convert the proof of Lemma~\ref{lem:fund-intersection} into an efficient algorithm.

  In order to find two fundamental cycles that have common edges, pick an edge $t$ of $T$, determine the two trees in the forest $T \setminus \{t\}$ and then iterate through the edges $e \in E$ to detect those edges that connect the two trees in $T \setminus \{t\}$.  There will be at least two such edges by 3-edge-connectedness of $G$.
  This procedure gives a pair $C$, $C'$ of intersecting fundamental cycles.  The proof of Lemma~\ref{lem:fund-intersection} continues by explaining how to decrease the number of edges in the intersection $C \cap C'$ by exchanging one of the two fundamental cycles $C, C'$ with another one.  Every exchange is based on deletion of two edges $e, e'$ from $T$ and looking at edges that connect the trees in the forest $T \setminus \{e, e'\}$.  Every cycle exchange can be carried out in time $O(m)$ using similar ideas.

  The algorithm emerging from Lemma~\ref{lem:fund-intersection} produces a sequence $(C_i, C_i')$, $i = 1, \ldots, s$ of pairs of fundamental cycles with $\{t\} := C_s \cap C_s' \subsetneq \ldots \subsetneq C_1 \cap C_1'$, spending $O(m)$ time units for each pair.
  After contraction of the edge $t$ in the main iteration, one can reuse the pairs $(C_i / t, C_i' / t)$, $i < s$ of contracted cycles in the following iterations.  With this approach, every pair of cycles so computed is used as a pair $\ci(e_i, T_i)$, $\ci(f_i, T_i)$ for some $i$.  Thus, the algorithm spends $O(m)$ time units per edge of $T$, and this amounts to the total running time $O(mn)$.

  Now suppose $G$ is a general connected graph with spanning tree $T$.  By Lemma~\ref{lem:3-edge-conn-reduction}, a cosimplification $\hat{G}$ of $G$ may be constructed in time $O(mn)$.  We further require that $\hat{G}$ be constructed so that only edges of $T$ are contracted, which is possible because each non-trivial series class of $G$ contains at most one edge outside of $T$.  Under this construction, the edges in $T \intersect \hat{G}$ form a spanning forest of the cosimplification.

  Let $G_i = (V_i, E_i), i = 1, \ldots, k$ be the connected components of $\hat{G}$ with $m_1, \ldots, m_k$ edges and $n_1, \ldots, n_k$ vertices respectively, and let $T_i = T \intersect E_i$ be the spanning tree of $G_i$ induced by $T$ under the cosimplification.
  Each $G_i$ is 3-edge-connected, so a semi-fundamental lattice cycle basis $\cycles_i$ of $G_i$ with respect to $T_i$ may be constructed in time $O(m_i n_i)$.  By the second part of Lemma~\ref{lem:3-edge-conn-reduction}, these lattice cycle bases may be lifted to a lattice cycle basis of $G$ in time $O(mn)$, for a total computational time of $O(mn + \sum_i m_i n_i) = O(mn)$.

  Last, we argue that this lifted lattice cycle basis is semi-fundamental with respect to $T$.  If $\hat{C} \in \cycles_i$ and $C$ is the lifting of $\hat{C}$ to $G$, then $\hat{C} \intersect T_i = C \intersect T$ because the edges contracted to form $\hat{G}$ were all edges of $T$.  Hence fundamental and semi-fundamental cycles of $T_i$ in $G_i$ are lifted to fundamental and semi-fundamental cycles of $T$ in $G$.
  Additionally, because the edges $E \setminus T$ are given by $\Union_i E_i \setminus T_i$, each fundamental $\ci_G(e, T)$ is given by the lifting of a fundamental cycle $\ci_{G_i}(e_i, T_i)$ for some $i$ and some $e_i \in E_i$.
\end{proof}

The following observation highlights a potentially useful property of the lattice cycle bases produced by the above algorithm: the lengths of cycles in a semi-fundamental basis are controlled by the diameter of the underlying spanning tree.

\begin{corollary}
	\label{cor:bdd-cycle-lengths}
  Let $G = (V, E)$ be a connected graph with $m$ edges and $n \geq 2$ vertices, and let $T$ be a spanning tree of $G$.  Then a lattice cycle basis of $G$ may be constructed in time $O(mn)$ such that each cycle has length at most $2 \diam(T)$.
\end{corollary}

\begin{proof}
	Any semi-fundamental lattice cycle basis with respect to $T$ has the desired property.  In particular, a fundamental cycle of $T$ has length at most $\diam(T) + 1$, and a semi-fundamental cycle, as the symmetric difference of two intersecting fundamental cycles, has length at most $2\diam(T)$.
\end{proof}

\subsection{Lattice Cycle Bases by Topological Extension}
\label{subsec:top-ext-basis-alg}

We now present a different approach for the construction of lattice cycle bases.  This approach is dynamic in the sense that lattice bases are built up for a sequence of successively larger graph minors, and the extensions at each step may be chosen independently.  Throughout this section, if $F \subseteq E$ are sets and $x \in \Q^E$, then we write $\proj{F}{x}$ for the standard projection of $x$ onto $\Q^F$.

\begin{definition}
  \label{def:top-one-edge-ext}
  For graphs $G$ and $H$, we write $H \xrightarrow{e} G$, and say that $G$ is a \defn{topological one-edge extension} of $H$, if $G$ is obtained from $H$ by one of the following operations:
  \begin{enumerate}[label={(\Alph*)}]
    \item \label{type:0-divide} A new edge $e$ is added between existing, possibly equal, vertices $a$ and $b$ of $H$.

    \item \label{type:1-divide} An edge $f$ of $H$ is divided into two edges $f_1, f_2$ by a new vertex $a$, and a new edge $e$ is added between $a$ and an existing vertex $b$ of $H$.

    \item \label{type:2-divide} Two distinct edges $f$ and $g$ of $H$ are each divided into two edges $f_1, f_2$ and $g_1, g_2$ by new vertices $a$ and $b$ respectively, and a new edge $e$ is added between $a$ and $b$.
  \end{enumerate}
\end{definition}

We say that the \defn{type} of the topological one-edge extension $H \xrightarrow{e} G$ is one of \ref{type:0-divide}, \ref{type:1-divide}, or \ref{type:2-divide}, depending on which of the above operations $G$ is derived from.  From the definition, we see that $H$ is a graph minor of $G$ obtained by deleting the new edge $e$ and contracting one edge from each split pair, depending on the type of the extension.

\begin{example}
  The complete graph $K_4$ on four vertices is a topological one-edge extension of the 3-edge bond graph.  In particular, the extension is of type \ref{type:2-divide}, and can be realized by picking $f$ and $g$ to be any pair of distinct edges in the 3-edge bond.
\end{example}

For a topological one-edge extension $H \xrightarrow{e} G$, the cycle lattice of $H$ naturally embeds into the cycle lattice of $G$.  Specifically, $H$ is obtained from $G \setminus e$ by contracting zero, one, or two edges, and each contracted edge is part of a non-trivial series class of $G \setminus e$.  The isomorphism of cycle lattices of $H$ and $G \setminus e$ then follows from Lemma \ref{lem:series-contractions}, and the cycle lattice of $G \setminus e$ canonically embeds into the cycle lattice of $G$ because $\cycles(G \setminus e) \subseteq \cycles(G)$.

The embedding of cycle lattices may be described by a linear map $\Q^{E(H)} \to \Q^{E(G)}$, which depends on the type of the extension as follows:
\begin{align*}
  \ref{type:0-divide}:
  &\hspace{2.5mm} x \mapsto
  \begin{bmatrix}
    x \\
    0
  \end{bmatrix}
  \!
  \begin{array}{c}
    {\scriptstyle E(H)} \\
    {\scriptstyle e}
  \end{array}
  &\ref{type:1-divide}:
  &\hspace{2.5mm} x \mapsto
  \begin{bmatrix}
    \proj{E(H) \setminus f}{x} \\
    x_f \\
    x_f \\
    0
  \end{bmatrix}
  \!
  \begin{array}{c}
    {\scriptstyle E(H) \setminus f} \\
    {\scriptstyle f_1} \\
    {\scriptstyle f_2} \\
    {\scriptstyle e}
  \end{array}
\end{align*}
\begin{align*}
  \ref{type:2-divide}:
  &\hspace{2.5mm} x \mapsto
  \begin{bmatrix}
    \proj{E(H) \setminus \{f, g\}}{x} \\
    x_f \\
    x_f \\
    x_g \\
    x_g \\
    0
  \end{bmatrix}
  \!
  \begin{array}{c}
    {\scriptstyle E(H) \setminus \{f, g\}} \\
    {\scriptstyle f_1} \\
    {\scriptstyle f_2} \\
    {\scriptstyle g_1} \\
    {\scriptstyle g_2} \\
    {\scriptstyle e}
  \end{array}
\end{align*}

Topological one-edge extension can be seen to be compatible with 3-edge-connectivity by the following result, whose proof is straightforward.

\begin{lemma}
	\label{lemma:top-ext-3ec}
	If $H \xrightarrow{e} G$ and $H$ is 3-edge-connected, then $G$ is 3-edge-connected as well.
\end{lemma}

In fact, 3-edge-connectivity of a graph can be characterized in terms of topological one-edge extensions.  The following result is essentially the case $l = 1$ of a theorem of Mader \cite{mader1978reduction} characterizing $(2l+1)$-edge-connected graphs, and can be viewed as a counterpart of Tutte's 1966 theorem characterizing 3-vertex-connectivity.  See \cite[Thm.~7.13]{frank1995connectivity} and \cite[Thm.~3.2.2]{diestel_graph_2017} for details.

\begin{proposition}[Growing a 3-edge-connected graph from a single vertex]
	\label{growing-three-edge-connected}
	A graph $G$ is 3-edge-connected if and only if
  \[
    G_0 \xrightarrow{e_1} G_1 \xrightarrow{e_2} \cdots \xrightarrow{e_k} G_k = G
  \]
  holds for some graphs $G_0, \ldots, G_k$, where $G_0$ is a single vertex.
\end{proposition}

We call a sequence of topological one-edge extensions as in the above a \defn{topological extension sequence} for $G$.  The following describes an algorithm to efficiently produce such a sequence.

\begin{thm:top-ext-seq-alg}
	Let $G$ be a 3-edge-connected graph with $n$ vertices and $m$ edges.  Then a topological extension sequence for $G$ exists, and can be constructed in time $O(mn)$.
\end{thm:top-ext-seq-alg}

\begin{proof}
  If $H$ is any graph, let $\Top(H)$ denote the \emph{topological representative} of $H$, defined as the graph obtained from $H$ by suppressing vertices of degree 2.  We will produce a topological extension sequence for $G$ presented as a decomposition of $E(G)$ into a sequence of edge-disjoint paths $P_0, \ldots, P_k$, where the graphs $G_i \defeq \Top(P_0 \union \ldots \union P_i)$ satisfy
  \[
    G_0 \xrightarrow{e_1} G_1 \xrightarrow{e_2} \cdots \xrightarrow{e_k} G_k = G.
  \]

	To start, let $P_0$ be any vertex of $G$, and let $P_1$ be a path starting and ending at the vertex of $P_0$.  Throughout, let $H_i \defeq P_0 \union \ldots \union P_i$.  For $i \geq 2$, suppose $i-1$ paths $P_0, \ldots P_{i-1}$ and corresponding subgraphs $H_0, \ldots, H_{i-1}$ have already been constructed.  If $H_{i-1} = G$, then the algorithm terminates.
  Otherwise, $E(G) \setminus E(H_{i-1})$ is nonempty, and because $G$ is connected, there is a vertex $v_i$ of $H_{i-1}$ which is incident to an edge not contained in $H_{i-1}$.
  We then use depth-first search to generate a path $P_i \subseteq G \setminus H_{i-1}$ starting at $v_i$, ending at some vertex of $H_{i-1}$, and containing no edges or interior vertices in $H_{i-1}$.

  If $\deg_{H_{i-1}}(v_i) > 2$ or if $H_{i-1}$ is a cycle graph, then any path $P_i$ using only edges and interior vertices outside $H_{i-1}$ is sufficient.  If $\deg_{H_{i-1}}(v_i) = 2$ and $H_{i-1}$ is not a cycle graph, then $v_i$ is an interior vertex of a path $Q_i$ with end vertices of degree at least 3 in $H_{i-1}$, and with interior vertices of degree exactly 2 in $H_{i-1}$.  In this case, we search for a path $P_i$ starting at $v_i$ and ending at some vertex $w_i$ of $H_{i-1}$ which is not also an interior vertex of $Q_i$.
  This ensures that $P_i$ induces a topological one-edge extension of $G_{i-1}$ of type either \ref{type:1-divide} or \ref{type:2-divide}, and does not connect two new vertices on the same edge in $G_{i-1}$.
  Such a path exists because otherwise the removal of the two edges of $Q_i$ incident with its endpoints would disconnect $G$, contradicting 3-edge-connectedness.

  As an optimization, after each path is completed in this way, inspect the edges of $G$ for edges incident to two vertices of degree at least 3 in $H_i$.  If $e$ is such an edge, then the additional path $P$ consisting only of the edge $e$ may be added to the output sequence.  The additional paths of this form do not require depth-first search in $G$, and the overall computational cost of adding these paths can be seen to be $O(mn)$.

  The correctness of the algorithm follows from Lemma \ref{lemma:top-ext-3ec}.  For the running time, we note that the number of iterations requiring a depth-first search is at most $3n$.  This follows because each path $P_i$ formed by a depth-first search either contains a vertex outside of $H_{i-1}$, or has an endpoint in $H_{i-1}$ of degree 2.  In either of these cases, some vertex of $G$ with degree less than 3 in $H_{i-1}$ has strictly larger degree in $H_i$, and this can occur at most $3n$ times throughout the construction.  Thus the search-based steps are computed in time $O(mn)$.  Because the single-edge paths between degree 3 vertices are also computed in time $O(mn)$, the result follows.
\end{proof}

In the following we provide an explicit construction to extend a basis of the cycle lattice from a 3-edge-connected graph to a topological one-edge extension.  This provides the primary tool which we will use to inductively construct lattice cycle bases from scratch.

\begin{proposition}
	\label{prop:growing-basis}
	Let $H$ be a 3-edge-connected graph, and let $G$ be a topological one-edge extension of $H$ with $m$ edges.  Then a set of cycles exists whose indicator vectors extend any basis of $\lattice(\cycles(H))$ to a basis of $\lattice(\cycles(G))$.  Such a set can be computed in time $O(m)$.
\end{proposition}

\begin{proof}
	Let $\basis = \{b_1, \ldots, b_m\}$ be a lattice basis of $\lattice(\cycles(H))$.  We argue by cases for the three possible types of the topological one-edge extension $H \xrightarrow{e} G$.  We will use the notation for such extensions described in Definition \ref{def:top-one-edge-ext}, and we will write $\Phi$ for the natural linear map embedding $\lattice(\cycles(H))$ into $\lattice(\cycles(G))$.
  In each case, once an extending collection of cycles $\cycles' \subseteq \cycles(G)$ of the correct cardinality has been determined, we verify that the collection $\Phi(\basis) \union \set{\indVect{C}}{C \in \cycles'}$ is a basis of $\lattice(\cycles(G))$ by checking that the new collection has determinant equal to $\clDet{G}$.

  For type \ref{type:0-divide}, a cycle $C$ in $G$ which contains $e$ can be constructed in time $O(m)$ by finding a simple path $P$ between $a$ and $b$ in $H$.  Then we have
  \begin{align*}
    \det (\Phi(\basis), \indVect{C})
    &= \det \left[ \!\!
    \begin{array}{ccc|c}
      b_1 & \cdots & b_m  & \chi_P \\
      \hline
      0 & \cdots & 0 & 1
    \end{array} \!\!
    \right]
    \!
    \begin{array}{c}
      {\scriptstyle E(H)} \\
      {\scriptstyle e}
    \end{array} \\
    &= \det (\basis).
  \end{align*}
  Thus the determinant of the new collection is equal to $\clDet{H} = \clDet{G}$, and we can let $\cycles' = \{C\}$.

  For type \ref{type:1-divide}, we can construct $C_1, C_2 \in \cycles(G)$ satisfying $C_i \cap \{f_1, f_2, e\} = \{f_i, e\}$ in time $O(m)$ as follows.
  If $v_i$ is the endpoint of $f$ in $H$ incident to $f_i$ in $G$, then a simple path $P_i$ can be found between $v_i$ and $b$ in the connected graph $H \setminus f$, and we can set $C_i = P_i \union \{f_i, e\}$.
  Setting $\Delta = \det( \Phi(\basis), \indVect{C_1}, \indVect{C_2} )$, we have
  \begin{align*}
    \Delta
    &= \det \left[ \!\!
    \begin{array}{lcl|ll} \proj{E(H) \setminus f}{b_1} & \cdots & \proj{E(H) \setminus f}{b_m} & \indVect{P_2} &  \indVect{P_1} \\
		\proj{f}{b_1} & \cdots & \proj{f}{b_m} & 0 & 1 \\
    \hline
    \proj{f}{b_1} & \cdots & \proj{f}{b_m} & 1 & 0 \\
    \phantom{(}0 & \cdots & \phantom{(}0 & 1 & 1
		\end{array} \!\!
    \right]
    \!
    \begin{array}{c}
      {\scriptstyle E(H) \setminus f} \\
      {\scriptstyle f_1} \\
      {\scriptstyle f_2} \\
      {\scriptstyle e}
    \end{array} \\
    &= \det (\basis) \cdot \det\! \begin{bmatrix*}[r] 1 & -1 \\ 1 & 1 \end{bmatrix*},
  \end{align*}
  where the last line is by subtracting the row indexed by $f_1$ from the row indexed by $f_2$ and taking determinants along the block diagonal.
  In this case, the determinant of the new collection is given by $2\clDet{H} = \clDet{G}$, so the collection $\cycles' = \{C_1, C_2\}$ satisfies the desired conditions.

  For type \ref{type:2-divide}, we construct three cycles $C_1$, $C_2$, and $C$ satisfying $C_i \cap \{f_1, f_2, g_1, g_2, e\} = \{f_i, g_i, e\}$ for $i = 1, 2$, and $C \cap \{f_1, f_2, g_1, g_2, e\} = \{f_1, g_2, e\}$.
  For $i = 1, 2$, let $v_i$ be the endpoint of $f$ incident to $f_i$ in $G$, and let $w_i$ be the endpoint of $g$ incident to $g_i$ in $G$.  Since $H$ is 3-edge-connected, the deletion $H \setminus \{f, g\}$ remains connected.
  Thus we can find simple paths $P_1$, $P_2$, $P$ in $H \setminus \{f, g\}$ in time $O(m)$ where $P_i$ joins $v_i$ and $w_i$ and $P$ joins $v_1$ and $w_2$.  Then we can set $C_i = P_i \union \{ f_i, g_i, e \}$ and $C = P \union \{f_1, g_2, e\}$.
  Setting $\Delta = \det( \Phi(\basis), \indVect{C_1}, \indVect{C_2}, \indVect{C} )$, we have
  \begin{align*}
    \Delta
    &= \det \left[ \!\!
    \begin{array}{lcl|lll}
      \proj{E(H) \setminus \{f, g\}}{b_1}
        & \cdots
        & \proj{E(H) \setminus \{f, g\}}{b_m}
        & \indVect{P_1}
        & \indVect{P_2}
        & \indVect{P} \\
  		\proj{f}{b_1} & \cdots & \proj{f}{b_m} & 1 & 0 & 1 \\
      \proj{g}{b_1} & \cdots & \proj{g}{b_m} & 1 & 0 & 0 \\
      \hline
  		\proj{f}{b_1} & \cdots & \proj{f}{b_m} & 0 & 1 & 0 \\
      \proj{g}{b_1} & \cdots & \proj{g}{b_m} & 0 & 1 & 1 \\
      \phantom{(}0  & \cdots & \phantom{(}0  & 1 & 1 & 1 \\
		\end{array} \!\!
    \right]
    \!
    \begin{array}{c}
      {\scriptstyle E(H) \setminus \{f, g\}} \\
      {\scriptstyle f_1} \\
      {\scriptstyle g_1} \\
      {\scriptstyle f_2} \\
      {\scriptstyle g_2} \\
      {\scriptstyle e}
    \end{array} \\
    &= \det (\basis) \cdot \det\!
    \begin{bmatrix*}[r]
      -1 & 1 & -1 \\
      -1 & 1 &  1 \\
       1 & 1 &  1
     \end{bmatrix*},
  \end{align*}
  where the last line is by subtracting the rows indexed by $f_1$ and $g_1$ from the rows indexed by $f_2$ and $g_2$ and taking determinants along the block diagonal.  In this case, the determinant of the new collection is given by $4 \clDet{H} = \clDet{G}$, so the collection $\cycles' = \{C_1, C_2, C\}$ satisfies the desired conditions.
\end{proof}

If $G_0 \xrightarrow{e_1} G_1 \xrightarrow{e_2} \cdots \xrightarrow{e_k} G_k$ is a topological extension sequence, define a \defn{compatible chain} of lattice cycle bases to be a nested sequence $\cycles_0 \subset \cycles_1 \subset \cdots \subset \cycles_k$, where $\cycles_i$ is a lattice cycle basis of $G_i$ for each $i$.  The following aggregates the preceding results to give an algorithm which produces a topological extension sequence and a compatible chain of lattice cycles bases of a 3-edge-connected graph $G$.  In particular, this produces a lattice cycle basis of $G$.

\begin{thm:top-ext-comp-chain}
  Let $G$ be a 3-edge-connected graph with $m$ edges and $n$ vertices.  Then a topological extension sequence of $G$ and a compatible chain of lattice cycle bases can be constructed in time $O(mn)$.
\end{thm:top-ext-comp-chain}

\begin{proof}
  From Theorem~\ref{thm:top-ext-seq-alg}, a topological extension sequence $G_0 \xrightarrow{e_1} G_1 \xrightarrow{e_2} \cdots \xrightarrow{e_k} G_k = G$ of $G$ can be computed in time $O(mn)$, and such an extension sequence has length $O(m)$.  Inductively we produce a lattice cycle basis $\cycles_i$ of $G_i$ for each $i$ by extending $\cycles_{i-1}$.

  As an auxiliary data structure, we maintain a spanning tree $T_i$ of $G_i$ at each step of the induction as follows.  For each edge $x$ divided in two by the topological extension $G_{i-1} \xrightarrow{e_i} G_i$, if $x \in T_{i-1}$, then include both of the resulting edges in $T_i$, and if $x \notin T_{i-1}$, then include only one of the resulting edges in $T_i$.  In particular, $e_i \notin T_i$ for any $i$.

  For each topological one-edge extension $G_{i-1} \xrightarrow{e_i} G_i$ of type~\ref{type:1-divide} or~\ref{type:2-divide}, Proposition~\ref{prop:growing-basis} shows how to construct $\cycles_i$ from $\cycles_{i-1}$ in time $O(m)$.  Any topological one-edge extension of these types adds at least one new vertex, so there are at most $O(n)$ such extensions in the sequence, and the overall time needed for computing these lattice cycle bases is $O(mn)$.

  For topological one-edge extensions of type~\ref{type:0-divide}, an extending cycle can be produced from any path between the endpoints of $e_i$ excluding this new edge.  In particular, we make use of the path between the endpoints of $e_i$ in the spanning tree $T_i$, which can be computed in time $O(n)$.  Since there are $O(m)$ such extensions in the extension sequence, the overall time needed for computing these lattice cycle bases is also $O(mn)$.
\end{proof}

As for Theorem~\ref{thm:semi-fund-basis}, this $O(mn)$ algorithm for producing a lattice cycle basis for 3-edge-connected graphs extends by Lemma~\ref{lem:3-edge-conn-reduction} to an $O(mn)$ algorithm for arbitrary connected graphs.
It is interesting to note that the extensions at each step of the above construction are independent of the other steps, which implies a multiplicative mode of growth for the number of lattice cycle bases of a 3-edge-connected graph in terms of its number of edges.

\section{Applications to Linear Hulls}
\label{sec:applications-lin-hulls}

The results discussed so far in this paper focus on the Abelian group structure of the cycle lattice, whose properties are closely connected to the combinatorics of the underlying graph.  We now present some consequences of these results to the $A$-linear hull of cycles of a graph $G$ with respect to an Abelian group $A$.
In the following, we view $\linhull_A(\cycles(G))$ as a subgroup of $A^E \simeq \Z^E \tensor_{\Z} A$.

\begin{thm:lin-hull-groups}
	Let $G = (V, E)$ be a 3-edge-connected graph with $m$ edges and $n$ vertices, and let $A$ be an Abelian group.  Then
	\[
		\linhull_A(\cycles(G)) \simeq (2A)^{n-1} \oplus A^{m-n+1}.
	\]
\end{thm:lin-hull-groups}

\begin{proof}
  Let $T$ be a spanning tree of $G$, and let $\phi_0 : A^E \to A^E$ be defined by
  \[
    \phi_0: a \mapsto \sum_{t \in T} a_t \indVect{t} + \sum_{e \in E \setminus T} a_e \indVect{\ci(e, T)}.
  \]
  By Proposition \ref{prop:simple-basis}, the vectors $\cycles_T = \set{\indVect{\ci(e, T)}}{e \in E \setminus T}$ with the vectors $X_T = \set{2 \indVect{t}}{t \in T}$ give a lattice basis of $\lattice(\cycles(G))$.  Consequently, for $a \in A^T$ and $b \in A^{E \setminus T}$,
	\[
		\phi_0(2a, b) = \sum_{t \in T} a_t ( 2 \indVect{t}) + \sum_{e \in E \setminus T} b_e \indVect{\ci(e, T)}
	\]
  is an element of $\linhull_A(\cycles(G))$, and in particular $\phi_0$ restricts to a homomorphism $\phi : (2A)^T \oplus A^{E \setminus T} \to \linhull_A(\cycles(G))$.

  Because $X_T \union \cycles_T$ is a basis of $\lattice(\cycles(G))$, we conclude that $\phi$ is surjective.
  To show that $\phi$ is injective, suppose $a \in A^T$ and $B \in A^{E \setminus T}$ such that $\phi(2a, b) = 0$.  Each coordinate $e \in E \setminus T$ has a nonzero value in the sum $\phi(2a, b)$ only for the summand corresponding to the cycle indicator vector $\indVect{\ci(e, T)}$, so we conclude that $b = 0$.
  Subsequently, each coordinate $t \in T$ has a nonzero value in the remaining sum $\phi(2a, 0)$ only for the summand corresponding to the indicator vector $\indVect{t}$, so we likewise have $2a = 0$.
\end{proof}

For finite Abelian groups, the previous result gives a simple method to determine the primary decomposition of $\linhull_A(\cycles(G))$ from the primary decomposition of $A$.
Each $p$-primary group $C_{p^k}$, $p \neq 2$ in the decomposition of $A$ introduces $m$ copies in the primary decomposition of $\linhull_A(\cycles(G))$, and each $2$-primary group $C_{2^k}$ introduces $m-n+1$ copies of $C_{2^k}$ along with $n-1$ copies of $C_{2^{k-1}}$.

As another application of Theorem~\ref{thm:lin-hull-groups}, we may apply the result to an arbitrary field to obtain the following.

\begin{thm:lin-hull-vector-sp}
  Let $G = (V, E)$ be a 3-edge-connected graph with $m$ edges and $n$ vertices, and let $K$ be a field of characteristic $p$.  Then $\linhull_K(\cycles(G))$ is a $K$-vector space of dimension
	\[
		\dim_K\!\big(\linhull_K(\cycles(G))\big)
    =
		\begin{cases}
			m, & \text{if } p \ne 2,
			\\ m-n+1, & \text{if } p = 2.
		\end{cases}
	\]
  If $p \neq 2$, then any lattice basis of $\lattice(\cycles(G))$ reduces modulo $p$ to a linear basis of $\linhull_K(\cycles(G))$.  If $p = 2$, then any basis of the classical binary cycle space maps to a linear basis of $\linhull_K(\cycles(G))$ under the natural inclusion map.
\end{thm:lin-hull-vector-sp}

\begin{proof}
	Let $U = \linhull_K(\cycles(G))$.  By Theorem~\ref{thm:lin-hull-groups}, the $K$-dimension of $U$ is given by $\dim_K\!\big((2K)^{n-1} \oplus K^{m-n+1}\big)$.  When $K$ has characteristic 2, we have $2K = \{0\}$, and otherwise, $2K = K$.

  For the vector space bases, the argument of Theorem~\ref{thm:lin-hull-groups} shows that the image of the lattice basis $\basis$ of Proposition \ref{prop:simple-basis} in $U$ generates the linear hull.  If $p \neq 2$, then $\basis$ has as many vectors as the dimension, and thus forms a vector space basis of $U$.
  If $\basis'$ is a lattice basis of $\lattice(\cycles(G))$, then its image in $U$ is a $\Z$-invertible linear transformation of $\basis$, and thus also forms a vector space basis.

  If $p = 2$, then the $n-1$ vectors $2\indVect{t}$ of $\basis$ map to 0 in $U$, and the remaining $m-n+1$ vectors $\basis_0$ form a basis of $U$, again by dimension considerations.
  The collection $\basis_0$ forms a basis of the binary cycle space of $G$ by Proposition \ref{prop:classical-cycle-space}, so any basis $\basis_0'$ of the binary cycle space is an invertible transformation of $\basis_0$ over $\Z / 2\Z \subseteq K$, and thus likewise maps to a vector space basis of $U$.
\end{proof}

\bibliographystyle{abbrv}
\bibliography{BibCircuits}

\end{document}